\documentclass[11pt]{amsart}

\usepackage[margin=1in]{geometry}

\usepackage{amsmath, amssymb, amsfonts, amsthm, mathtools}

\usepackage{graphicx}
\usepackage{xcolor}

\usepackage{tikz}
\usetikzlibrary{3d,calc}
\usepackage{tikz-cd}

\usepackage{algorithm}
\usepackage{algpseudocode}

\usepackage{enumitem}
\usepackage{parskip}

\usepackage[colorlinks=true,linkcolor=blue,citecolor=blue,urlcolor=blue]{hyperref}
\usepackage[capitalise,noabbrev]{cleveref}

\newtheorem{theorem}{Theorem}[section]
\newtheorem{lemma}[theorem]{Lemma}
\newtheorem{corollary}[theorem]{Corollary}

\newtheorem{claim}[theorem]{Claim}

\theoremstyle{definition}
\newtheorem{definition}[theorem]{Definition}

\newtheorem{example}[theorem]{Example}
\newtheorem{remark}[theorem]{Remark}

\DeclareMathOperator{\depol}{depol}
\DeclareMathOperator{\reg}{reg}
\DeclareMathOperator{\pd}{pd}
\DeclareMathOperator{\height}{ht}

\title{On licci squarefree monomial ideals}

\author{Om Prakash Bhardwaj}
\address{Chennai Mathematical Institute, Siruseri, Tamilnadu~603103. India}
\email{omprakash@cmi.ac.in, opbhardwaj95@gmail.com}

\author{Trung Chau}
\address{Chennai Mathematical Institute, Siruseri, Tamilnadu~603103. India}
\email{chauchitrung1996@gmail.com}

\author{Omkar Javadekar}
\address{Chennai Mathematical Institute, Siruseri, Tamilnadu~603103. India}
\email{omkarj@cmi.ac.in, omkarjavadekar@gmail.com}

\keywords{Monomial ideals, Path ideals, Licci ideals, Regularity, Cohen--Macaulay ring}

\subjclass[2020]{14M06, 13C40, 13H10, 05E40}

\begin{document}

\begin{abstract}
We study the licci property for several classes of squarefree monomial ideals arising from graphs and related combinatorial structures. We characterize licci bi-Cohen–Macaulay squarefree monomial ideals, complementary edge ideals, $t$-path ideals of cycles, and  $(t-1)$-suspensions of graphs. Consequently, the full list of licci path ideals of trees is obtained. This work extends the known classification of licci edge ideals to a broader family of path ideals.
\end{abstract}
\maketitle

\section{Introduction}
The notion of linkage was first formally introduced by Peskine--Szpiro in \cite{Peskine-Szpiro}. Given a commutative Noetherian ring $R$ and ideals $I, J\subseteq R$, we say that $I$ and $J$ are \emph{directly linked}, denoted $I\sim J$, if there is a regular sequence $\underline{x}=x_1, \ldots, x_c\in I\cap J$ such that $I=(\underline{x})\colon_R\, J$ and $J= (\underline{x}) \colon_R\, I$. In general, the ideals $I$ and $J$ are said to be \emph{linked} if there is a sequence of ideals $I=I_0, I_1, \ldots, I_r=J$ such that $I_i$ is directly linked to $I_{i+1}$ for every $i$. It is easy to see that linkage induces an equivalence relation on the set of all ideals of $R$ in which two ideals are in the same equivalence class if they are linked. We call the equivalence classes of this relation as \emph{linkage classes}. An ideal $I$ is said to be \emph{licci} if $I$ belongs to the linkage class of a complete intersection ideal. The theory of linkage is a useful tool in classifying ideals and algebraic varieties (see \cite{EneRinaldoTerai_liccibinomial,Herzog_Crelle, Huneke_AJM, Huneke_Inv, HunekeUlrich_Duke, Kimura-Terai-Yoshida, KustinMiller_Gorenstein, KustinMiller_tightGorenstein, KustinMillerUlrich_pureRes, Rao_liason}).

The classification of licci ideals arising from combinatorial objects such as graphs and simplicial complexes has been a topic of significant interest (see \cite{EneRinaldoTerai_liccibinomial, Kimura-Terai-Yoshida, NagelRomer_glicci}). In \cite{Kimura-Terai-Yoshida}, Kimura--Terai--Yoshida characterized the licci property among edge ideals of finite simple graphs, the most prominent class of squarefree monomial ideals studied in combinatorial commutative algebra.  A goal of this article is to extend the licci characterization to other well studied classes of squarefree monomial ideals. We refer to subsequent sections for unexplained concepts and terminology.

It is important to note that licci ideals are Cohen--Macaulay~(\cite{Peskine-Szpiro}), which serves as our starting point. A squarefree monomial ideal $I$ in a polynomial ring $R$ over a field is called \emph{bi-Cohen--Macaulay} if $R/I$ and $R/I^{\vee}$ are both Cohen--Macaulay. This property arises from the study of Bernstein-Gel'fand-Gel'fand correspondence and coherent sheaves \cite{biCM-Crupi,biCM-Gunnar,biCM-Herzog}. In our first main result, we fully characterize licci bi-Cohen--Macaulay squarefree monomial ideals.

\begin{theorem}[\protect{\Cref{thm:licci-biCM}}]\label{thm:licci-bi-CM-intro}
 Let $I$ be a bi-Cohen--Macaulay squarefree monomial ideal in a polynomial ring $S$ over a field. Let $\mathfrak{m}$ denote the homongeneous maximal ideal of $S$. Then $I_{\mathfrak{m}}$ is licci in $S_{\mathfrak{m}}$ if and only if $\operatorname{ht}(I) \leq 2$ or $I$ is generated by variables.     
\end{theorem}

Despite being a recent terminology,  complementary edge ideals have sparked great interests in the community of commutative algebraists with many algebraic and homological invariants computed and properties fully characterized \cite{ ficarra2026stablesetassociatedprimes,Ficarra-Moradi,hibi2025complementaryedgeideals}. Our next result provides a full characterization of licci complementary edge ideals, adding to the growing literature around this concept.

\begin{theorem}[\protect{\Cref{thm:licci-complementary}}]
 Let $G$ be a finite simple graph and $I_c(G)$ denote its complementary edge ideal. Set $S=\Bbbk[V(G)]$, where $\Bbbk$ is a field, and let $\mathfrak{m}$ denote the homongeneous maximal ideal of $S$. Then $I_c(G)_{\mathfrak{m}}$ is licci in $S_\mathfrak{m}$  if and only if $G$ is a forest or a triangle.     
\end{theorem}

The last class of ideals we study are path ideals. Introduced by Conca--De Negri in their study of $M$-sequences \cite{ConcaDeNegri}, path ideals serve as a natural generalization of edge ideals of graphs. Their algebraic and homological properties are largely open due to the simple fact  that it is usually not straightforward to obtain all of their monomial minimal generators. For example, only partial results are known about when path ideals have linear resolution, or are Cohen--Macaulay (\cite{Banerjee2015,path-ideals-CM, path-linear-res}). On the other hand, it is known that for any $t\geq 2$, the $t$-path ideal of $\Sigma_t G$ is Cohen--Macaulay, where $\Sigma_t G$ denotes the suspension of length $t-1$ of $G$ (\cite[proof of Theorem~3.8]{path-ideals-CM}). These form a large class of graphs (e.g., the map $G\mapsto \Sigma_tG$ is a injective map between graphs). We remark that $\Sigma_2 G$ is more well-known as its \emph{whiskered graph}. Our last main result is a full characterization of path ideals of graphs, assuming that the graphs are obtained through suspension and that the base field is infinite.

\begin{theorem}[\protect{\Cref{thm:licci-path-ideals-suspension}}]
    Let $G$ be a finite simple graph and $t\geq 2$ a positive integer. Let $\mathfrak{m}$ denote the homogeneous maximal ideal of $T=\Bbbk[V(\Sigma_t G)]$. Consider the following statements.
\begin{enumerate}[label={\rm(\arabic*)}]
        \item $P_t(\Sigma_t G)_{\mathfrak{m}}$ is licci in $T_\mathfrak{m}$.
        \item $G$ is a star graph with possible isolated vertices.
        \item Either $t\leq 3$ and $G$ is a star graph with possible isolated vertices, or $t>3$ and $G$ has at most one edge.
    \end{enumerate}
    Then $(3) \Longrightarrow (1) \Longrightarrow (2)$. Moreover, if $\Bbbk$ is infinite, then $(1) \Longrightarrow (3)$.
\end{theorem}

The article is organized as follows. In Section~\ref{sec:prem}, we recall the necessary definitions and known results relevant to this work, and derive several useful consequences. In Sections~\ref{sec:bi-CM} and \ref{sec:complementary}, we characterize bi-Cohen–Macaulay licci monomial ideals and licci complementary edge ideals, respectively. The remaining two sections are devoted to licci path ideals. In Section~\ref{sec:path}, we characterize licci path ideals of suspensions of graphs and of licci path ideals of cycles.

\section*{Acknowledgements}

O. P. Bhardwaj is supported by the ANRF National Postdoctoral Fellowship. The authors acknowledge the support from the Infosys Foundation.

\section{Preliminaries}\label{sec:prem}
Let  $R$ be a commutative Noetherian ring. Let $I$ and $J$ be ideals of $R$, we say that $I$ and $J$ are \emph{directly linked}, denoted $I\sim J$, if there is a regular sequence $\underline{x}=x_1, \ldots, x_c\in I\cap J$ such that $I=(\underline{x})\colon_R\, J$ and $J= (\underline{x}) \colon_R\, I$. In general, the ideals $I$ and $J$ are said to be \emph{linked} if there exists a sequence $I_0, I_1, \ldots, I_r$ of ideals such that
$$I=I_0 \sim I_1 \sim \ldots \sim I_r=J.$$ 
If $I$ is linked to $J$ then we say that $I$ and $J$ are in same linkage class. 
\begin{definition}
    Let $R$ be a commutative Noetherian local ring and $I$ be an ideal of $R$. The ideal $I$ is called \emph{licci} if it is in the linkage class of a complete intersection ideal. 
\end{definition}

\begin{remark}\label{rem:hgt2-and3}
It is well-known (see e.g.~\cite[Lemma 2.2]{Kimura-Terai-Yoshida}) that if  $R$ is a regular local ring and $I\subseteq R$ is an ideal. Then we have the following.
\begin{enumerate}
    \item If $I$ is Cohen--Macaulay with $\height(I)=2$, then $I$ is licci in $R$.
    \item If $I$ is Gorenstein with $\height(I)=3$, then $I$ is licci in $R$.
\end{enumerate}
\end{remark}

\begin{remark}\label{rem:height-1}
    Let $S$ be a polynomial ring over a field, $\mathfrak{m}$ its homogeneous maximal ideal, and $I$ a homogeneous ideal of $S$. Since $S$ is a  unique factorization domain, any prime ideal of height 1 is principal. Therefore if $I$ is Cohen-Macaulay with $\height(I)=1$, then $I$ is also principal. In particular, $I_\mathfrak{m}$ is licci in $R_{\mathfrak{m}}$.
\end{remark}

\begin{remark}
    Let $S$ be a polynomial ring over a field, $\mathfrak{m}$ its homogeneous maximal ideal, and $I$ a homogeneous ideal of $S$. Let $I$ be a homogeneous ideal of $S$ then $\height(I) = \height(I_{\mathfrak{m}})$ (e.g., see~\cite[Proposition 1.5.15]{Bruns-Herzog}). We will use this fact without repeated reference.
\end{remark}

\begin{remark}\label{rem:Nagata}
    Let $S$ be a polynomial ring over a field, $\mathfrak{m}$ its homogeneous maximal ideal, and $I$ a homogeneous ideal of $S$. If $I$ is a homogeneous ideal such that $I_\mathfrak{m}$ is licci, then $I_{\mathfrak{m}}$ is Cohen-Macaulay \cite{Peskine-Szpiro}. By Nagata's conjecture (Theorem) \cite{localization-CM}, $I$ is Cohen-Macaulay. We will use this fact throughout the article.
\end{remark}

Let $S=\Bbbk[x_1, \ldots, x_n]$ be a standard graded polynomial ring and $I$ be a homogeneous ideal of $S$. The \emph{projective dimension of $S/I$}, denoted $\pd(S/I)$, is the largest integer $i$ for which $\operatorname{Tor}_i^S(S/I, \Bbbk)\neq 0$.  The $(i,j)^{th}$ Betti number of $S/I$ over $S$, denoted $\beta_{i,j}^S(S/I)$ is defined as 
$$\beta_{i,j}^S(S/I)= \dim_{\Bbbk} \operatorname{Tor}_i^{S}(S/I, \Bbbk)_j.$$
The \emph{Castelnuovo--Mumford regularity} of $S/I$, or regularity for short, denoted $\reg(S/I)$ is defined as 
$$\reg(S/I)= \max \{ j -i  \mid \beta_{i,j}^S(S/I)\neq 0\}.$$

\begin{lemma}\label{lem:reg-CM}
    Let $I$ be a homogeneous ideal in a polynomial ring $S$ such that $S/I$ is Cohen--Macaulay. Then
    \[
    \reg (S/I) = \max\{ j-\pd (S/I) \mid \beta_{\pd (S/I), j }(S/I)\neq 0\}. 
    \]
    In other words, regularity is achieved at the end of the minimal free resolution.
\end{lemma}
\begin{proof}
    Follows from \cite[Exercise 20.19]{Eis-book}.
\end{proof}

Note that if $I$ is a homogeneous ideal of $S=\Bbbk[x_1, \ldots, x_n]$ generated by a regular sequence $f_1, \ldots,f_r$ with $\deg(f_i)=d_i>0$, then the Koszul complex on $f_1, \ldots, f_r$ gives a minimal free resolution of $S/I$ over $S$, and we have $\reg(S/I)=\left( \sum\limits_{i=1}^r d_i\right) -r  $. 
The following result due to Huneke and Ulrich \cite{HunekeUlrich} says that if $I_\mathfrak{m}$ is licci in $S_{\mathfrak{m}}$, then $\reg(S/I)$ exhibits a similar behaviour. Let $\alpha(I)$ denote the minimum degree of a nonzero polynomial in $I$.
\begin{lemma}\label{lem:not-licci-criteria}
    Let $I$ be a homogeneous ideal in a polynomial ring $S$ such that $S/I$ is Cohen--Macaulay, and $\mathfrak{m}$ be the homogeneous maximal ideal of $S$. If we have
    \[
    \reg (S/I)\leq (\alpha(I)-1)\pd (S/I) -\alpha(I),
    \]
    then $I_\mathfrak{m}$ is not licci in $S_{\mathfrak{m}}$.
\end{lemma}
\begin{proof}
    Note that if $S/I$ is not  Cohen--Macaulay, then $I$ is not licci in $S$. So, we assume that $S/I$ is Cohen--Macaulay (Remark~\ref{rem:Nagata}). Let $g = \pd(S/I)$, and
    \[0 \longrightarrow \bigoplus_{i=1}^{b_g} S(-n_{gi}) \longrightarrow \bigoplus_{i=1}^{b_{g-1}} S(-n_{(g-1)i}) \longrightarrow \cdots \longrightarrow \bigoplus_{i=1}^{b_1} S(-n_{1i}) \longrightarrow S \longrightarrow S/I \longrightarrow 0\]
    be a graded minimal free resolution of $S/I$. Note that by definition, $\alpha(I)= \mathrm{min}\{n_{1i}\}_i$, and by Lemma~\ref{lem:reg-CM}, $\mathrm{max}\{n_{gi}\}_i=\reg (S/I) + g$. Thus our inequality in the hypothesis can be rewritten as follows:
    \[
    \mathrm{max}\{n_{gi}\}_i-g \leq (\mathrm{min}\{n_{1i}\}_i-1) g - \mathrm{min}\{n_{1i}\}_i.
    \]
    Equivalently, we obtain
    \[
    \mathrm{max}\{n_{gi}\}_i \leq (\mathrm{min}\{n_{1i}\}_i-1) g - \mathrm{min}\{n_{1i}\}_i+g = (g-1)\mathrm{min}\{n_{1i}\}_i.
    \] 
    By \cite[Corollary 5.13]{HunekeUlrich}, this implies that $I_{\mathfrak m}$ is not licci in $S_{\mathfrak m}$, as desired.
\end{proof}

For the rest of the section let $I \subseteq S$ be an $\mathfrak{m}$-primary monomial ideal. We recall a criterion, due to Huneke and Ulrich~\cite{HunekeUlrichMonomial}, for $I$ to be licci. Note that every $\mathfrak{m}$-primary monomial ideal $I$ can be written uniquely in standard form,
$$I = (x_1^{a_1},x_2^{a_2}, \ldots, x_n^{a_n}) + I^{\#}$$
where $I^{\#}$ is generated by monomials that together with $\{x_1^{a_1},x_2^{a_2}, \ldots, x_n^{a_n}\}$ generate $I$ minimally. Note that if $I$ is a complete intersection then $I^{\#} = 0$. If $I$ is not a complete intersection then $I^{\#}$ can be written uniquely as $I^{\#} = \mathbf{x}^BK$, where $\mathbf{x}^B = x_1^{b_1}\cdots x_n^{b_n}$ is a monomial and $K$ a monomial ideal of height at least two. Set $I^{\{0\}} = I$, and define
$$ I^{\{1\}} = (x_1^{a_1-b_1}, \ldots, x_n^{a_n-b_n}) + K$$
If $I$ is a complete intersection, set $I^{\{1\}} = S$. For $k>1$, define
inductively $I^{\{k\}} = (I^{\{k-1\}})^{\{1\}}$, provided that $I^{\{k-1\}} \neq S$. Using this notation, the following theorem characterizes the licci property of $\mathfrak{m}$-primary monomial ideals.

\begin{theorem}[{\cite[Theorem 2.6]{HunekeUlrichMonomial}}]\label{thm-Huneke-Ulrich}
Let $S = \Bbbk[x_1,\ldots,x_n]$ and $I \subseteq S$ be an $\mathfrak{m}$-primary monomial ideal of $S$. The following conditions are equivalent.
\begin{enumerate}[label=\rm (\roman*)]

\item $I_{\mathfrak{m}}$ is licci in $S_{\mathfrak{m}}$.

\item $(I^{\{k\}})^{\#}$ has height at most one whenever $I^{\{k\}} \neq S$. 

\item $I^{\{k\}} = S$ for some $k$.

\end{enumerate}
\end{theorem}

\begin{example}
    Consider $S = \Bbbk[x_1,x_2,x_3]$ and $I = (x_1^2,x_2^2,x_3^2,x_1x_2,x_2x_3)$. Write $I = (x_1^2,x_2^2,x_3^2) + x_2(x_1,x_3).$ Thus, $I^{\{1\}} = (x_1^2,x_2,x_3^2) + (x_1,x_3) = (x_1,x_2,x_3).$ Since $I^{\{1\}}$ is a complete intersection, $I^{\{2\}} = S$. Therefore, by the above theorem $I_{\mathfrak{m}}$ is licci in $S_{\mathfrak{m}}$.
\end{example}

\section{Licci bi-Cohen--Macaulay ideals}\label{sec:bi-CM}

Throughout the rest of the article let $[n]$ denote the set $\{1,\dots, n\}$ for any positive integer $n$.

Let $I$ be a squarefree monomial ideal of $S=\Bbbk[x_1,\dots, x_n]$. For each subset $F\subseteq [n]$, set $x_F\coloneqq \prod_{i\in F}x_i$. Let 
\[
I=\bigcap_{i=1}^m (x_j\ \vert \ j\in F_i)
\]
be the unique primary decomposition of $I$, where $F_1,\dots, F_m$ are some subsets of $[n]$. The \emph{Alexander dual} of $I$, denoted $I^\vee$, is defined as
\[
I^\vee \coloneqq (x_{F_i}\ \vert\ i\in [m] ).
\]
It is well known that $(I^\vee)^\vee=I$. By definition, we have $\alpha(I)=\height(I^\vee)$ and $\alpha(I^\vee)=\height(I)$.

Recall that $I$ is called \emph{bi-Cohen--Macaulay} if $I$ and $I^\vee$ are both Cohen-Macaulay. We prove our first main result, the full characterization of licci bi-Cohen--Macaulay ideals.

\begin{theorem}\label{thm:licci-biCM}
    Let $I$ be a bi-Cohen--Macaulay squarefree monomial ideal in a polynomial ring $S$ over a field. Let $\mathfrak{m}$ be the homogeneous maximal ideal of $S$. Then $I_{\mathfrak{m}}$ is licci in $S_{\mathfrak{m}}$ if and only if $\operatorname{ht} I \leq 2$ or $I$ is generated by variables. 
\end{theorem}
\begin{proof}
    If $I$ is generated by variables then $I$ is a complete intersection, and thus $I_\mathfrak{m}$ is licci, as desired. On the other hand, if $\height(I)\leq 2$, then $I_\mathfrak{m}$ is licci by Remarks~\ref{rem:hgt2-and3} and \ref{rem:height-1}, as desired.
    
    Now suppose that $I$ is licci. Then by Lemma~\ref{lem:not-licci-criteria}, we have 
    \[
    \reg (S/I) > (\alpha(I)-1)\pd (S/I) - \alpha(I) = (\alpha(I) -1)\operatorname{ht}I - \alpha(I),
    \]
    where $\pd(S/I)=\height(I)$ is due to $S/I$ being Cohen-Macaulay.
    By \cite{Terai}, we have $\reg(S/I) = \pd(S/I^\vee) - 1.$ Therefore, we get  
\[\pd(S/I^\vee) - 1 > (\alpha(I)-1)\height(I) - \alpha(I),\]
and thus $(\alpha(I)-1)(\height(I) -1) < \alpha(I)$. If $\height(I)\leq 2$, then we are done. Now we can assume that $\height(I)\geq 3$. The inequality $(\alpha(I)-1)(\height(I) -1) < \alpha(I)$ then implies that $\alpha(I)=1$, i.e., $I$ is generated by variables, as desired. 
\end{proof}

Gitler and Valencia \cite{Gitler-Valencia} proved that for if $I$ is an unmixed squarefree monomial ideal such that $I$ is generated in degree 2 and every variable of $S$ appears in at least one monomial minimal  generator of $I$,  then 
\[
\height(I) \geq \left\lceil \frac{n}{2} \right\rceil.
\]
A generalization of this result to all squarefree monomial ideals remains open. Using similar arguments as in the proof of Theorem~\ref{thm:licci-biCM}, we obtain the following result for licci squarefree monomial ideals, resembling Gitler-Valencia theorem.

\begin{theorem}
    Let $I$ be a licci squarefree  monomial ideal of $S=\Bbbk[x_1, \ldots, x_n]$.  Let $\mathfrak{m}$ be the homogeneous maximal ideal of $S$. If $I_\mathfrak{m}$ is licci in $S_{\mathfrak{m}}$, then 
    \[\height(I) \leq \left\lfloor \frac{n}{\alpha(I)}\right\rfloor + 1.\]
\end{theorem}
\begin{proof}
    Since $I_\mathfrak{m}$ is licci, by Lemma \ref{lem:not-licci-criteria}, we get 
    \[ \reg(S/I) > (\alpha(I)-1)\pd(S/I)-\alpha(I) = (\alpha(I)-1) \operatorname{ht} (I)\alpha(I).\]
By \cite{Terai}, we have $\reg(S/I)=\pd(S/I^\vee)-1$. This gives $\pd(S/I^\vee) - 1> (\alpha(I) -1)\height(I) -\alpha(I)$. Let $p$ be the projective dimension of $S/I^\vee$. Then, the function $f: \{0,1,\ldots,p\} \longrightarrow \mathbb{Z}$ given by $f(i) = \mathrm{min}\{j \mid \beta_{ij}(S/I^\vee) \neq 0\}$ is strictly increasing for $0 \leq i \leq p$. We have $f(1) = \alpha(I^\vee)$, and moreover $f(p)$ is at most the degree of the least common multiple of all minimal monomial generators of $I^\vee$, which is less than or equal to $n$. Therefore, we get $\pd(S/I^\vee) \leq n-\height(I) +1$. This implies that $n-\height(I) +1 > (\height(I)-1)(\alpha(I)-1)$, and hence $\height(I) < \frac{n}{\alpha(I)} + 1$. This concludes the proof.
\end{proof}

\section{Licci complementary edge ideals}\label{sec:complementary}
Let $G=(V(G),E(G))$ be a finite simple graph on $V(G)=\{x_1,\dots, x_n\}$, where $n\geq 3$. The complementary edge ideal of $G$ is defined as 
\[ I_{c}(G) = \left( x_1\cdots x_n/x_ix_j \mid \{i,j\} \in E(G) \right).\]

\begin{example}
    Let $G$ be the following graph.\\
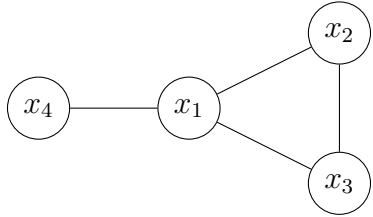
   \begin{center}
       \begin{tikzpicture}[
    every node/.style={circle, draw, fill=white, minimum size=8mm, font=\large}
]
  \node (a) at (0, 0)  {$x_1$};
  \node (b) at (2, 1)  {$x_2$};
  \node (c) at (2, -1) {$x_3$};
  \node (d) at (-2, 0) {$x_4$};

  \draw (a) -- (b);
  \draw (a) -- (c);
  \draw (b) -- (c);
  \draw (a) -- (d);
\end{tikzpicture}
   \end{center}
   The corresponding complementary edge ideal is 
   \[
   I_c(G)=\left(\frac{x_1x_2x_3x_4}{x_1x_2},\frac{x_1x_2x_3x_4}{x_1x_3},\frac{x_1x_2x_3x_4}{x_1x_4},\frac{x_1x_2x_3x_4}{x_2x_3}\right)=(x_3x_4,x_2x_4,x_2x_3,x_1x_4).\]
\end{example}

Recall that $G$ is called \emph{complete}, in which case $G$ is denoted by $K_n$,  if $\{x_i,x_j\}\in E(G)$ for any $i,j\in [n]$ where $i\neq j$. A graph $G$ is called an \emph{$n$-cycle}, in which case $G$ is denoted by $C_n$, if after a change of labelling on the variables, we have $E(G)=\{\{x_i,x_{i+1}\}\mid i\in [n]\}$,  where we adopt the convention $x_{n+1}=x_1$. A graph $G$ is called a \emph{forest} if it has no cycle subgraph. A connected forest is called a \emph{tree}.

In the following, we recall the Cohen--Macaulay property of complementary edge ideals.

\begin{theorem}[{\cite[Theorem~2.8]{Ficarra-Moradi}}]\label{CMcomplementaryedgeideals}
 Let $G$ be a finite simple graph on $n$ vertices without isolated vertices, where $n\geq 3$ is a positive integer. Then $I_{c}(G)$ is Cohen--Macaulay if and only if $G$ is a complete graph or a forest.
\end{theorem}

As it turns out, most of these ideals are also licci.

\begin{theorem}\label{thm:licci-complementary}
    Let $G$ be a finite simple graph on $V(G)=\{x_1,\dots, x_n\}$, where $n\geq 3$ is a positive integer. Set $S=\Bbbk[V(G)]$ and let $\mathfrak{m}$ denote its homogeneous maximal ideal. Then, $I_{c}(G)_\mathfrak{m}$ is licci in $S_{\mathfrak{m}}$ if and only if $G$ is a triangle ($K_3$) or a forest. 
\end{theorem}
\begin{proof}
    Suppose $I_{c}(G)_\mathfrak{m}$ is licci. We want to show that $G$ is either $K_3$ or a forest. Note that $I_c(G)$ is Cohen-Macaulay by Remark~\ref{rem:Nagata}. If $G$ has an isolated vertex, say $x_1$, then $(x_1)$ is a minimal prime of $I_c(G)$. Since $I_c(G)$ is Cohen-Macaulay, it is principal by Remark~\ref{rem:height-1}. In particular, this implies that $G$ has one edge, and thus $G$ is a forest, as desired. Now we can assume that $G$ has no isolated vertices. Then    $G$ is either $K_n$ or a forest by Theorem~\ref{CMcomplementaryedgeideals}. It now suffices to show that $I_{c}(K_n)_\mathfrak{m}$ is not licci for any $n\geq 4$. Indeed, we have that $\height(I_{c}(K_n)) = 3$ (\cite[Corollary 2.3]{Ficarra-Moradi}) and $I_{c}(K_n)$ has linear resolution (\cite[Corollary~4.8]{Ficarra-Moradi}). Therefore, by Lemma~\ref{lem:not-licci-criteria}, $I_{c}(K_n)_\mathfrak{m}$ is not licci for any $n\geq 4$, as desired. 
    
    Conversely, suppose that $G$ is $K_3$ or a forest. We want to show that $I_{c}(G)_\mathfrak{m}$ is licci. Indeed, if $G=K_3$, then $I_c(G)=(x_1,x_2,x_3)$, a complete intersection. Now we can assume that $G$ is a forest. Since $G$ is not complete,  $\height (I_{c}(G)) = 2$ (\cite[Corollary 2.3]{Ficarra-Moradi}). Thus, by Remark~\ref{rem:hgt2-and3}, $I_{c}(G)_\mathfrak{m}$ is licci, as desired.   
\end{proof}

\section{Licci path ideals of graphs}\label{sec:path}

In this section, we will characterize the licci $t$-path ideals of $(t-1)$-suspensions of other graphs, and licci $t$-path ideals of cycles. Throughout this section, let $G=(V(G),E(G))$ be a finite simple graph with $V(G)=\{x_1,\dots, x_n\}$, and $t\geq 2$ a positive integer. Recall that a \emph{path of length $t-1$} of $G$ is an ordered sequence $x_{i_1},\dots, x_{i_t}$ of $t$ distinct vertices of $G$ such that $\{x_{i_j},x_{i_{j+1}}\}\in E(G)$ for any $i\in [t-1]$.

The \emph{suspension of length $t-1$} of $G$, denoted by $\Sigma_t G$, is a finite simple graph with the vertex set
\[
V(G)\sqcup \{ x_{ij}\mid i\in [n], j\in [t-1] \}
\]
and the edge set
\[
E(G)\sqcup \{\{x_i,x_{i1}\},\ \{x_{ij},x_{i,j+1}\}\mid i\in [n], j\in [t-2] \}.
\]
Pictorially, $\Sigma_t G$ is the graph $G$ with a path of length $t-1$ attached to each vertex of $G$. The suspension of length $1$ of a graph $G$ is also known as the \emph{whiskered graph} of $G$. 

\begin{example}\label{ex:suspension}
    Let $G$ be a triangle with edges $\{\{x_1,x_2\},\{x_1,x_3\},\{x_2,x_3\}\}$. Below are the pictures of $G, \Sigma_2G$, and $\Sigma_3 G$, respectively.
    
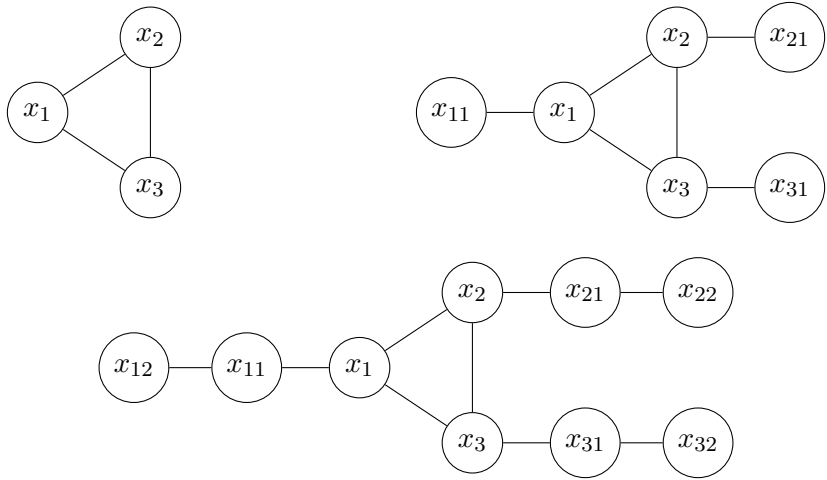
\begin{center}
        \begin{tikzpicture}[
    every node/.style={circle, draw, fill=white, minimum size=7mm, font=\normalsize}
]

\node (a1) at (0, 0)    {$x_1$};
\node (b1) at (1.5, 1)  {$x_2$};
\node (c1) at (1.5, -1) {$x_3$};

\draw (a1) -- (b1);
\draw (a1) -- (c1);
\draw (b1) -- (c1);

\node (a2) at (7, 0)    {$x_1$};
\node (b2) at (8.5, 1)  {$x_2$};
\node (c2) at (8.5, -1) {$x_3$};
\node (u2) at (5.5, 0)  {$x_{11}$};
\node (v2) at (10, 1)    {$x_{21}$};
\node (w2) at (10, -1)   {$x_{31}$};

\draw (a2) -- (b2);
\draw (a2) -- (c2);
\draw (b2) -- (c2);
\draw (a2) -- (u2);
\draw (b2) -- (v2);
\draw (c2) -- (w2);

\end{tikzpicture}

\bigskip

\begin{tikzpicture}[
    every node/.style={circle, draw, fill=white, minimum size=7mm, font=\normalsize}
]

\node (a3) at (3, 0)    {$x_1$};
\node (b3) at (4.5, 1)  {$x_2$};
\node (c3) at (4.5, -1) {$x_3$};
\node (u3) at (1.5, 0)  {$x_{11}$};
\node (v3) at (6, 1)    {$x_{21}$};
\node (w3) at (6, -1)   {$x_{31}$};
\node (x3) at (0, 0)    {$x_{12}$};
\node (y3) at (7.5, 1)  {$x_{22}$};
\node (z3) at (7.5, -1) {$x_{32}$};

\draw (a3) -- (b3);
\draw (a3) -- (c3);
\draw (b3) -- (c3);
\draw (a3) -- (u3);
\draw (u3) -- (x3);
\draw (b3) -- (v3);
\draw (v3) -- (y3);
\draw (c3) -- (w3);
\draw (w3) -- (z3);

\end{tikzpicture}
    \end{center}
\end{example}

The \emph{$t$-path ideal} of $G$, denoted by $P_t(G)$, is defined to be the ideal generated by squarefree monomials $x_{i_1}\cdots x_{i_t},$ 
where $x_{i_1},\dots, x_{i_t}$ is a path of length $t-1$ in $G$,
in $\Bbbk[x_i\mid i\in [n]]$. If $G$ has no path of length $t-1$, we adopt the convention that $P_t(G)=(0)$. The $2$-path ideals are also known as \emph{edge ideals}.

\begin{example}\label{ex:suspension-2}
    Let $G$ be as in Example~\ref{ex:suspension}. We then compute $P_3(G)$, $P_3(\Sigma_2G)$, and $P_3(\Sigma_3G)$. This computation is the same as determining all paths of length $2$ in the three graphs. We have
    \begin{align*}
        P_3(G)&=(x_1x_2x_3),\\
        P_3(\Sigma_2G)&=P_3(G)+(x_1x_2x_{11},x_1x_3x_{11},x_1x_2x_{21},x_1x_3x_{31},x_2x_3x_{21},x_2x_3x_{31}),\\
        P_3(\Sigma_3G)&=P_3(\Sigma_2G)+(x_1x_{11}x_{12},x_2x_{21}x_{22},x_3x_{31}x_{32}).
    \end{align*}
\end{example}

It is known to experts that the edge ideals of whiskered graphs are Cohen--Macaulay \cite{Villarreal}, and the same holds for the $t$-path ideals of $(t-1)$-suspension of graphs. We will include the details for completion (see Lemma~\ref{lem:depol-Artinian}, Corollary~\ref{cor:CM-path-ideals}).

For the rest of this section, we set
\begin{align*}
    R&=\Bbbk[V(G)]=\Bbbk[x_i\mid i\in [n]],\\
    T&=\Bbbk[V(\Sigma_tG)]=\Bbbk[x_i, x_{ij}\mid i\in [n],j\in [t-1]].
\end{align*}

Assume that $P_t(\Sigma_t G)$ is generated by squarefree monomials $m_1,\dots, m_q$ in $T$. By setting the variables in $\{x_{ij}\mid j\in [t-1]\}$ to $x_i$ for each $i$, we obtain the corresponding monomials $m_1',\dots, m_q'$ in $R$, which we denote by $\depol(m_1),\dots, \depol(m_q)$, correspondingly. The monomial ideal
\[
(\depol(m_1),\dots, \depol(m_q))\subseteq R
\]
is called a \emph{depolarization} of $P_t(\Sigma_t G)$, denoted by $\depol(P_t(\Sigma_t G))$. The reason for the name is due to an operation called \emph{polarization} (see, for example, \cite[Section 3.2]{MillerSturmfels}) and the fact that $P_t(\Sigma_t G)$ is a polarization of $\depol(P_t(\Sigma_t G))$ (\cite[Proof of Theorem 3.8]{path-ideals-CM}).

\begin{example}
    Let $G$ be as in Example~\ref{ex:suspension}. Recall from Example~\ref{ex:suspension-2} that we have
    \begin{multline*}
          P_3(\Sigma_3G)=(x_1x_2x_3)+(x_1x_2x_{11},x_1x_3x_{11},x_1x_2x_{21},x_1x_3x_{31},x_2x_3x_{21},x_2x_3x_{31})+\\
          (x_1x_{11}x_{12},x_2x_{21}x_{22},x_3x_{31}x_{32}).
    \end{multline*}
    By our definition, we have
    \[
    \depol(P_3(\Sigma_3G))=(x_1x_2x_3)+(x_1^2x_2,x_1^2x_3,x_1x_2^2,x_1x_3^2,x_2^2x_3,x_2x_3^2)+(x_1^3,x_2^3,x_3^3).
    \]
\end{example}

We recall some properties regarding depolarization.

\begin{lemma}[\protect{\cite[Corollary 1.6.3]{HerzogHibibook}}]\label{lem:polarization-properties}
    Let $G$ be a finite simple graph  and $t\geq 2$ a positive integer. The following holds.
\begin{enumerate}[{label=\rm (\alph*)}]
        \item $\pd (R/\depol(P_t(\Sigma_t G))) = \pd( T/P_t(\Sigma_t G))$;
        \item $\reg( R/\depol(P_t(\Sigma_t G))) = \reg( T/P_t(\Sigma_t G))$;
        \item $R/\depol(P_t(\Sigma_t G)) $ is Cohen--Macaulay if and only if   so is $ T/P_t(\Sigma_t G)$.
    \end{enumerate}
\end{lemma}

We recall the following.

\begin{lemma}[\protect{\cite[Proof of Theorem 3.8]{path-ideals-CM}}]\label{lem:depol-Artinian}
    Let $G$ be a finite simple graph with $V(G)=\{x_1,\dots, x_n\}$, for some $n\geq 1$, and $t\geq 2$ a positive integer. Then $x_i^t\in \depol(P_t(\Sigma_t G))$ for any $i\in [n]$. In particular, $R/\depol(P_t(\Sigma_t G))$ is Artinian, and hence Cohen--Macaulay.
\end{lemma}

\begin{proof}
    Since for each $i\in [n]$, the monomial $x_i\prod_{j=1}^{t-1}x_{ij}$ is a monomial generator in $P_t(\Sigma_t G)$, its depolarization
    \[
    \depol\left(x_i\prod_{j=1}^{t-1}x_{ij}\right)=x_i^t
    \]
    is in $\depol(P_t(\Sigma_t G))$. Therefore $\sqrt{\depol(P_t(\Sigma_tG))} = (x_1,\dots, x_n)$. This concludes the proof.
\end{proof}

The following follows immediately from Lemmas~\ref{lem:polarization-properties} and \ref{lem:depol-Artinian}.

\begin{corollary}\label{cor:CM-path-ideals}
    Let $G$ be a finite simple graph and $t\geq 2$ a positive integer. Then $T/P_t(\Sigma_t G)$ is Cohen--Macaulay.\qed
\end{corollary}

The $(t-1)$-suspension of graphs therefore forms a large class of graphs whose $t$-path ideals are Cohen--Macaulay.

We recall the formula for the regularity of an Artinian monomial ideal.

\begin{lemma}[\protect{\cite[Exercise 20.18]{Eis-book}}]\label{lem:formula-reg-Artinian}
    Let $I$ be an Artinian monomial ideal in a polynomial ring $S$ with the homogeneous maximal ideal $\mathfrak{m}$. Then $\reg(S/I)= \max \{\deg (f) \mid f \text{ is a monomial in } (I:_S\mathfrak{m})/I\}$. 
\end{lemma}


We next study the regularity of  $\depol(P_t(\Sigma_t G))$. Recall that $G$ is called an \emph{$n$-star} if after a relabelling of vertices, we have $E(G)=\{ \{x_1,x_2\}, \{x_1,x_3\},\dots, \{x_1,x_n\} \}$.

\begin{theorem}\label{thm:reg-inequality-equivalent}
    Let $G$ be a finite simple graph with $V(G)=\{x_1,\dots, x_n\}$, for some $n\geq 1$, and $t\geq 2$ a positive integer. Then $G$ is a star graph with possible isolated vertices if and only if
    \[
    \reg \left(T/P_t(\Sigma_t G)\right) > (t-1)n-t.
    \]
\end{theorem}

\begin{proof}
    Due to Lemma~\ref{lem:polarization-properties}, we can replace $P_t(\Sigma_t G)$ with its depolarization. Recall that the ideal $\depol(P_t(\Sigma_t G))$ is Cohen--Macaulay and Artinian (Lemma~\ref{lem:depol-Artinian}), and thus
    \[
    \pd \left(R/\depol(P_t(\Sigma_t G)\right) = \height  (\depol(P_t(\Sigma_t G))) = n.
    \] 
    Next we analyze the regularity. By Lemma~\ref{lem:formula-reg-Artinian}, we have $\reg \left(T/\depol(P_t(\Sigma_t G))\right)=\deg (F)$ for some monomial $F$ in $\depol(P_t(\Sigma_t G))\colon _R (x_i\mid i\in [n])$ but not in $\depol(P_t(\Sigma_t G))$. Set $F=x_1^{a_1}\cdots x_n^{a_n}$. For each $i\in [n]$, since $x_i^t\in \depol(P_t(\Sigma_t G))$ (see Proof of Lemma~\ref{lem:depol-Artinian}), we have  $a_i\leq t-1$. We have the following claim.
    \begin{claim}\label{clm:1-edge}
        If  $\{x_1,x_2\}$ is an edge of $G$, then $(x_1,x_2)^t\subseteq \depol(P_t(\Sigma_t G))$. Consequently, we have $a_1+a_2\leq t-1$.
    \end{claim}
    \begin{proof}[Proof of Claim~\ref{clm:1-edge}]
        Since  $\{x_1,x_2\}$ is an edge of $G$, the graph $\Sigma_t G$ contains the following subgraph
        
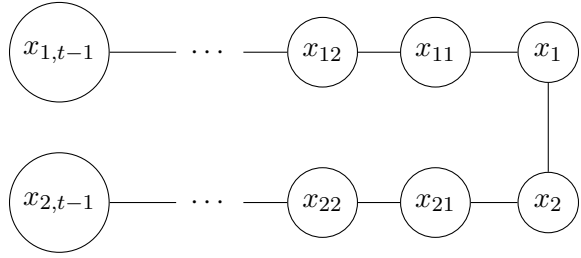
\begin{center}
            \begin{tikzpicture}[
    every node/.style={circle, draw, fill=white, minimum size=7mm, font=\normalsize}
]

\node (xtm1) at (0, 1)   {$x_{1,t-1}$};
\node (xdots1) [draw=none, fill=none] at (2, 1) {$\cdots$};
\node (x2)    at (3.5, 1) {$x_{12}$};
\node (x1)    at (5, 1)   {$x_{11}$};
\node (x0)    at (6.5, 1) {$x_1$};

\draw (xtm1) -- (xdots1);
\draw (xdots1) -- (x2);
\draw (x2) -- (x1);
\draw (x1) -- (x0);

\node (y0)    at (6.5, -1) {$x_2$};
\node (y1)    at (5, -1)   {$x_{21}$};
\node (y2)    at (3.5, -1) {$x_{22}$};
\node (ydots1) [draw=none, fill=none] at (2, -1) {$\cdots$};
\node (ytm1)  at (0, -1)  {$x_{2,t-1}$};

\draw (y0) -- (y1);
\draw (y1) -- (y2);
\draw (y2) -- (ydots1);
\draw (ydots1) -- (ytm1);

\draw (x0) -- (y0);

\end{tikzpicture}
        \end{center}        
        In particular, the monomials 
        \[
        x_1\left(\prod_{j=1}^{t-1} x_{1j}\right), \quad x_2\left(\prod_{j=1}^{t-1} x_{2j}\right), \quad   x_1x_2\left(\prod_{j=1}^k x_{1j} \right)\left(\prod_{j=1}^{t-2-k}x_{2j}\right),
        \]
        where $k\in [t-2]\cup \{0\}$, are minimal monomial generators of $P_t(\Sigma_t G)$. Therefore, their depolarization
        \[
        x_1^t,\quad x_2^t, \quad x_1^{k+1}x_2^{t-1-k}, 
        \]
        where $k\in [t-2]\cup \{0\}$, are minimal monomial generators of $\depol(P_t(\Sigma_t G))$. In other words, $(x_1,x_2)^t\subseteq \depol(P_t(\Sigma_t G))$. Therefore $f\notin \depol(P_t(\Sigma_t G))$ implies that $a_1+a_2\leq t-1$, as desired.
    \end{proof}

    Back to the main proof, suppose that $G$ is not a star graph with possible isolated vertices. Equivalently (we leave this straightforward exercise to the readers), $G$ either contains two disjoint edges or a triangle as a subgraph. Assume that $G$ contains two disjoint edges. Without loss of generality, assume that $\{x_1,x_2\},\{x_3,x_4\}$ are edges of $G$. Then by Claim~\ref{clm:1-edge}, we have $a_1+a_2\leq t-1$ and $a_3+a_4\leq t-1$. Therefore, we have
    \begin{align*}
        \sum_{i=1}^n a_i = (a_1+a_2)+(a_3+a_4) + \sum_{i=5}^n a_i &\leq (t-1)+(t-1) + \sum_{i=5}^n (t-1) \\
        &= (t-1)(n-2)\\
        &=(t-1)n - t - (t-2)\\
        &\leq (t-1)n -t,
    \end{align*}
   as desired. Now we can assume that $G$ contains a triangle. Without loss of generality, assume that $\{x_1,x_2\},\{x_2,x_3\},\{x_1,x_3\}$ are edges of $G$.

By Claim~\ref{clm:1-edge}, we have the three inequalities:
        \[
        a_1+a_2 \leq t-1, \quad a_2+a_3 \leq t-1, \quad a_1+a_3 \leq t-1.
        \]
        Therefore $a_1+a_2+a_3\leq \frac{3}{2}(t-1)$. We thus have
        \begin{align*}
            \sum_{i=1}^n a_i = (a_1+a_2+a_3) + \sum_{i=4}^n a_i &\leq \frac{3}{2}(t-1) + \sum_{i=4}^n (t-1)\\
            &=(t-1)(n-\frac{3}{2})\\
            &=(t-1)n - t - \frac{1}{2}(t-3)\\
        \end{align*}
        {If $t \geq 3$, we have $\sum_{i=1}^n a_i \leq (t-1)n - t$. If $t=2$, then we have $\sum_{i=1}^n a_i \leq (t-1)n-t + \frac{1}{2}$. Since $\sum_{i=1}^n a_i$ and $(t-1)n-t$ are integers, we get $\sum_{i=1}^n a_i \leq (t-1)n - t$,}
        as desired. In either case, we have
    \[
    \reg \left(R/\depol(P_t(\Sigma_t G))\right) \leq (t-1)n-t,
    \]
    as desired.

    Conversely, assume that $G$ is a star graph with possible isolated vertices. Without loss of generality, assume that
    \[
    E(G)=\{ x_ix_k \mid i\in [k-1] \}
    \]
    for some $k\in [n]$. If $k=1$ or $n=1$, i.e., $G$ is $n$ isolated vertices, then it is straightforward that
    \[
    \reg \left(R/\depol(P_t(\Sigma_t G))\right) = \reg \left(\Bbbk[x_i\mid i\in [n]] / (x_i^t\mid i\in [n])\right) = tn-n>  (t-1)n-t,
    \]
    as desired. Now assume that $k,n\geq 2$. We have the following claim.
    \begin{claim}\label{clm:socle-element}
        We have 
        \[
        (x_1\cdots \widehat{x_k}\cdots x_n)^{t-1}\in \left( \depol(P_t(\Sigma_t G))\colon _R (x_i\mid i\in [n]) \right)\setminus \depol(P_t(\Sigma_t G)).
        \]
        In particular, $\reg \left(R/\depol(P_t(\Sigma_t G))\right) \geq (t-1)(n-1)$.
    \end{claim}
    \begin{proof}[Proof of Claim~\ref{clm:socle-element}]
        The second statement follows from the first statement and Lemma~\ref{lem:formula-reg-Artinian}. 
        We have
        \[
        x_i (x_1\cdots \widehat{x_k}\cdots x_n)^{t-1} = x_i^t (\cdots) \in \depol(P_t(\Sigma_t G)) 
        \]
        for any $i\in [n]\setminus \{k\}$, by Lemma~\ref{lem:depol-Artinian}. On the other hand, since $x_k,x_1,x_{11},\dots, x_{1,t-2}$ form a path of length $t-1$ in $\Sigma_t G$, we have
        \[
        x_1^{t-1}x_k = \depol(x_k x_1x_{11}\cdots x_{1,t-2}) \in \depol(P_t(\Sigma_t G)),
        \]
        which implies that 
        \[
        x_k (x_1\cdots \widehat{x_k}\cdots x_n)^{t-1} = x_1^{t-1}x_k (\cdots) \in \depol(P_t(\Sigma_t G)).
        \]
        Therefore, we have
        \[
        (x_1\cdots \widehat{x_k}\cdots x_n)^{t-1}\in \left( \depol(P_t(\Sigma_t G))\colon _R (x_i\mid i\in [n]) \right).
        \]
        Moreover, due to the structure of $\Sigma_t G$, any path of length $t-1$ in it must either be \[\{x_j,x_{j,1},\dots, x_{j,t-1}\}\]
        for some $j\in [n]\setminus \{k\}$, or contain $x_k$. In other words, we have
        \[
        (x_1\cdots \widehat{x_k}\cdots x_n)^{t-1}\notin \depol(P_t(\Sigma_t G)),
        \]
        as desired.
    \end{proof}

    The desired inequality then follows immediately from Claim~\ref{clm:socle-element}. This concludes the proof.
\end{proof}

We are now in a position to prove the main theorem of the section, a partial characterization of licci $t$-path ideals, assuming that the graphs are $(t-1)$-suspension (of some other graphs).

\begin{theorem}\label{thm:licci-path-ideals-suspension}
    Let $G$ be a finite simple graph and  $t\geq 2$ be a positive integer. Let $\mathfrak{m}$ denote the homogeneous maximal ideal of $T=\Bbbk[V(\Sigma_tG)]$. Consider the following statements.
\begin{enumerate}[label={\rm(\arabic*)}]
        \item $P_t(\Sigma_t G)_{\mathfrak{m}}$ is licci in $T_\mathfrak{m}$.
        \item $G$ is a star graph with possible isolated vertices.
        \item Either $t=2 $ and $G$ is a star graph with possible isolated vertices, or $t>2$ and $G$ has at most one edge.
    \end{enumerate}
    Then $(3) \Longrightarrow (1) \Longrightarrow (2)$. Moreover, if $\Bbbk$ is infinite, then $(1)\implies (3)$.
\end{theorem}

\begin{proof}   
    Set $V(G)=\{x_1,\dots, x_n\}$, for some $n\geq 1$.
    
    \emph{(1)$\implies$ (2).} This follows from Lemma~\ref{lem:not-licci-criteria} and Theorem~\ref{thm:reg-inequality-equivalent}.  

    \emph{(3)$\implies$ (1).} The case $t=2$ follows from \cite[Theorem~3.7]{Kimura-Terai-Yoshida}. We can now assume that $t\geq 3$ and $G$ has at most one edge. If $G$ has no edge, then $P_t(\Sigma_t G)$ is a complete intersection, and thus the result follows. Now without loss of generality, assume that $E(G)=\{ x_1,x_2 \}$. 
    Set $x_{i0}\coloneqq x_i$ for each $i\in [n]$. It is straightforward that 
    \begin{multline*}
        P_t(\Sigma_t G) =
         \left(\left(\prod_{j=0}^{l-1} x_{1j} \right)\left(\prod_{j=0}^{t-1-l}x_{2j}\right)\ \middle\vert \  l\in [t]\cup \{0\} \right) +  \left(\prod_{j=0}^{t-1}x_{ij}\ \middle\vert \  i\in [n]\setminus [2] \right).
    \end{multline*}
   Set
   \begin{align*}
        I&\coloneqq P_t(\Sigma_t G),\\
        J&\coloneqq \left(\prod_{j=0}^{t-1}x_{ij}\ \middle\vert \  i\in [n] \right),\\
        I'&\coloneqq \left(\left(\prod_{j=1}^{l-1} x_{1j} \right)\left(\prod_{j=1}^{t-1-l}x_{2j}\right)\ \middle\vert \  l\in [t-1] \right) +  \left(\prod_{j=0}^{t-1}x_{ij}\ \middle\vert \  i\in [n]\setminus [2] \right),\\
        J'&\coloneqq \left(\prod_{j=1}^{t-1}x_{ij}\ \middle\vert \  i\in [2] \right)+ \left(\prod_{j=0}^{t-1}x_{ij}\ \middle\vert \  i\in [n]\setminus [2] \right).
    \end{align*}

    We have the following claim.

    \begin{claim}\label{clm:links}
        The two ideals $I$ and $I'$ are linked.
    \end{claim}
    \begin{proof}[Proof of Claim~\ref{clm:links}]
        It is clear that $J$ (resp, $J'$) is a complete intersection in $I$ (resp, $I'$) with the same height. It now suffices to show that $J\colon I=J'\colon I'$. Note that $\mathbf{a}\colon \mathbf{b}=\mathbf{a}\colon (\mathbf{b},f)$ for any two ideals $\mathbf{a}$ and $\mathbf{b}$ and any element $f\in \mathbf{a}$. Thus we have
        \begin{align*}
            J\colon I &= \left(\prod_{j=0}^{t-1}x_{ij}\ \middle\vert \  i\in [n] \right) \colon \left(\left(\prod_{j=0}^{l-1} x_{1j} \right)\left(\prod_{j=0}^{t-1-l}x_{2j}\right)\ \middle\vert \  l\in [t-1] \right) \\
            &= \left(\prod_{j=0}^{t-1}x_{ij}\ \middle\vert \  i\in [n] \right) \colon x_{10}x_{20}\left(\left(\prod_{j=1}^{l-1} x_{1j} \right)\left(\prod_{j=1}^{t-1-l}x_{2j}\right)\ \middle\vert \  l\in [t-1] \right)\\
            &=\left( \left(\prod_{j=1}^{t-1}x_{ij}\ \middle\vert \  i\in [2] \right)  + \left(\prod_{j=0}^{t-1}x_{ij}\ \middle\vert \  i\in [n]\setminus [2] \right) \right)\colon \left(\left(\prod_{j=1}^{l-1} x_{1j} \right)\left(\prod_{j=1}^{t-1-l}x_{2j}\right)\ \middle\vert \  l\in [t-1] \right)\\
            &=J'\colon I',
        \end{align*}
        as desired.
    \end{proof}
    For each $k\in [t]$, set
    \[
    I_k\coloneqq \left(\left(\prod_{j=0}^{l-1} x_{1j} \right)\left(\prod_{j=0}^{k-1-l}x_{2j}\right)\ \middle\vert \  l\in [k]\cup \{0\} \right) +  \left(\prod_{j=0}^{t-1}x_{ij}\ \middle\vert \  i\in [n]\setminus [2] \right).
    \]
    Observe that $I_t=I$, and $I_{t-2}$ is exactly $I'$ after a change of variables. Therefore, Claim~\ref{clm:links} implies that $(I_t)_\mathfrak{m}$ is licci if and only if so is $(I_{t-2})_\mathfrak{m}$. In fact, with similar arguments, one can argue that 
    \[
    (I_{k})_\mathfrak{m} \text{ is licci if and only if so is } (I_{k-2})_\mathfrak{m}
    \]
    for any $k\in [t]\setminus [2]$. Therefore, it suffices to show that $(I_{1})_\mathfrak{m}$ and $(I_{2})_\mathfrak{m}$ are licci. Indeed, $I_{1}$ itself is a complete intersection, while we have
    \begin{align*}
        &\ \ \ \ \left(\left(x_{10}x_{11}, \ x_{20}x_{21} \right)+ \left(\prod_{j=0}^{t-1}x_{ij}\ \middle\vert \  i\in [n]\setminus [2] \right) \right)\colon I_2\\
        &=\left(\left(x_{10}x_{11}, \ x_{20}x_{21} \right)+ \left(\prod_{j=0}^{t-1}x_{ij}\ \middle\vert \  i\in [n]\setminus [2] \right) \right) \colon \left( x_{10}x_{20} \right)\\
        &= \left(x_{11},  \ x_{21} \right)+ \left(\prod_{j=0}^{t-1}x_{ij}\ \middle\vert \  i\in [n]\setminus [2] \right), 
    \end{align*}
    a complete intersection, as desired.

    \emph{(3) $\Longleftarrow$ (1), assuming that $\Bbbk$ is infinite.} Since (1) implies (2), $G$ is a star graph with possible isolated vertices. It therefore suffices to show that if $t>2$ and $G$ is a star graph with at least two edges and possible isolated vertices, then $P_t(\Sigma_t G)_{\mathfrak{n}}$ is not licci in $T_\mathfrak{n}$. Set $E(G)=\{x_1x_r,\dots, x_{r-1}x_r\}$ for some $3\leq r\leq n$, and $x_{i0}\coloneqq x_i$ for each $i\in [n]$.
    
    Since $\Bbbk$ is infinite, by \cite[Remark 3.13]{polarization-glicci}, $P_t(\Sigma_t G)_{\mathfrak{n}}$ is licci in $T_\mathfrak{n}$ if and only if  $\depol(P_t(\Sigma_t G))_{\mathfrak{n}}$ is licci in $R_\mathfrak{m}$. By Lemma~\ref{lem:depol-Artinian}, the latter is an Artinian ideal, and by Theorem~\ref{thm-Huneke-Ulrich}   it therefore suffices to show that $\depol(P_t(\Sigma_t G))^{\{s\}}\neq R$ for any $s\geq 2$. Indeed, we have
    \begin{multline*}
        P_t(\Sigma_t G) = \left(\left(\prod_{j=0}^{l-1} x_{ij} \right)\left(\prod_{j=0}^{t-1-l}x_{rj}\right)\ \middle\vert \  i\in [r-1],\ l\in [t]\cup \{0\} \right)+\\
        \left(\left(\prod_{j=0}^{l-1} x_{ij} \right)(x_{r0})\left(\prod_{j=0}^{t-2-l}x_{i'j}\right)\ \middle\vert \ i,i'\in [r-1] \text{ where }i<i',  l\in [t-2] \right) +  \left(\prod_{j=0}^{t-1}x_{ij}\ \middle\vert \  i\in [n]\setminus [r] \right)
    \end{multline*}
    and hence
    \begin{align*}
        &\ \ \ \  \depol(P_t(\Sigma_t G)\\ &= \begin{multlined}[t]
            \left( x_{i}^{l} x_{r}^{t-l} \ \middle\vert \  i\in [r-1],\ l\in [t]\cup \{0\} \right)+
        \left( x_{i}^{l} x_{i'}^{t-1-l}x_r \ \middle\vert \ i,i'\in [r-1] \text{ where }i<i', ,  l\in [t-2] \right) + \\
        \left(x_{i}^t\ \middle\vert \  i\in [n]\setminus [r] \right).
        \end{multlined}\\
        &= \begin{multlined}[t]
            \left(x_{i}^t\ \middle\vert \  i\in [n] \right)+ \left( x_{i}^{l} x_{r}^{t-l} \ \middle\vert \  i\in [r-1],\ l\in [t-1] \right)+\\
        \left( x_{i}^{l} x_{i'}^{t-1-l}x_r \ \middle\vert \ i,i'\in [r-1]\text{ where }i<i',  l\in [t-2] \right)
        \end{multlined}\\
        &= \begin{multlined}[t]
            \left(x_{i}^t\ \middle\vert \  i\in [n] \right)+ x_r\Big(\left( x_{i}^{l} x_{r}^{t-1-l} \ \middle\vert \  i\in [r-1],\ l\in [t-1] \right)+\\
        \left( x_{i}^{l} x_{i'}^{t-1-l} \ \middle\vert \ i,i'\in [r-1] \text{ where }i<i',  l\in [t-2] \right)\Big).
        \end{multlined}
    \end{align*}
    We then have
    \begin{align*}
        &\ \ \ \  I\coloneqq \depol(P_t(\Sigma_t G)^{\{1\}})\\
        &=  \begin{multlined}[t]
            \left(x_{i}^t\ \middle\vert \  i\in [n]\setminus \{r\} \right) + (x_r^{t-1})+ \left( x_{i}^{l} x_{r}^{t-1-l} \ \middle\vert \  i\in [r-1],\ l\in [t-1] \right)+\\
        \left( x_{i}^{l} x_{i'}^{t-1-l} \ \middle\vert \ i,i'\in [r-1] \text{ where }i<i',  l\in [t-2] \right)
        \end{multlined}\\
        &=  \begin{multlined}[t]
            \left(x_{i}^{t-1}\ \middle\vert \  i\in [r] \right)+\left(x_{i}^{t}\ \middle\vert \  i\in [n]\setminus [r]\right) + \left( x_{i}^{l} x_{r}^{t-1-l} \ \middle\vert \  i\in [r-1],\ l\in [t-2] \right)+\\
        \left( x_{i}^{l} x_{i'}^{t-1-l} \ \middle\vert \ i,i'\in [r-1] \text{ where }i<i',  l\in [t-2] \right).
        \end{multlined}
    \end{align*}
    Since $r\geq 3$, the monomial minimal generators of $I^{\#}$ have no common factor. Thus by definition, we have
    \[
    \depol(P_t(\Sigma_t G)^{\{s\}} = I^{\{s-1\}}= I^{\{s-2\}} =\cdots = I^{\{1\}}= I\neq R
    \]
    for any $s\geq 2$, as desired. This concludes the proof.
\end{proof}

As a corollary, we obtain all trees whose path ideals are licci.

\begin{corollary}
    Let $G$ be a tree and $t\geq 2$ be a positive integer. Let $\mathfrak{m}$ denote the homogeneous maximal ideal of $R=\Bbbk[V(G)]$.
    If $\Bbbk$ is infinite, then $P_t(G)_\mathfrak{m}$ is licci in $R_{\mathfrak{m}}$ if and only if $G$ is a path of length $t-1$ or $2t-1$.
\end{corollary}

\begin{proof}
    If $G$ is a path of length $t-1$ or $2t-1$, then $G=\Sigma_tH$ where $H$ is one isolated vertex or an edge, respectively. Thus by Corollary~\ref{cor:CM-path-ideals}, $P_t(G)$ is Cohen-Macaulay. On the other hand, $P_t(G)_\mathfrak{m}$ being licci also implies that $P_t(G)$ is Cohen-Macaulay. Therefore, we can assume that $P_t(G)$ is Cohen-Macaulay in the first place. By \cite[Theorem~3.8]{path-ideals-CM}, $G$  must be $\Sigma_t H$ for some graph $H$. The result then follows from Theorem~\ref{thm:licci-path-ideals-suspension}.
\end{proof}

We fell short of fully characterizing licci path ideals. On the other hand, it is known that $P_2(G)_\mathfrak{m}$ being licci implies that $G$ is a unicyclic graph (graphs with at most one cycle) (\cite[Theorem 3.7]{Kimura-Terai-Yoshida}). We predict that the full characterization of licci path ideals will be a similar result. Towards that end, we also characterize licci path ideals of cycles. In the following remark, we recall the known results about the regularity and projective dimension of path ideals of cycles.

\begin{remark}[{\cite[Corollary 5.5]{AlilooeeFaridi}}]\label{rem:reg-for-cycle} Let $n, t, q$ and $d$ be integers such that $2 \leq t \leq n$, $n = (t + 1)q + d$, where $q \geq 0$, $0 \leq d \leq t$. Then we have

$$\pd(S/P_t(C_n)) = \begin{cases}
    2q+1 & {\text {if}\ \ } d\neq 0\\
    2q & {\text {if}\ \ } d=0. 
\end{cases} $$

and 

$$\reg(S/P_t(C_n)) = \begin{cases}
    (t-1)q+d-1 & {\text {if}\ \ } d\neq 0\\
    (t-1)q & {\text {if}\ \ } d=0. 
\end{cases} $$
\end{remark}

We are now ready to characterize licci path ideals of cycles. 

\begin{theorem}
     Let $n\geq t\geq 2$ be integers, and $S=\Bbbk[x_1, \ldots, x_n]$. Then the ideal $P_t(C_n)_{\mathfrak{m}}$ is  licci in $S_{\mathfrak{m}}$ if and only if $n \in \{ t+1, 2t+1\}$. 
\end{theorem}
\begin{proof}
    If $n=t$, then $P_t(C_n)=(x_1\cdots x_n)$ is principal, and hence $P_t(C_n)_\mathfrak{m}$ licci, as desired. Now we can assume that $n\geq t+1$. We first prove that if $n \not\in \{ t+1, 2t+1\}$, then the ideal $P_t(C_n)_{\mathfrak{m}}$ is not licci in $S_{\mathfrak{m}}$. Let $n \not\in \{t+1, 2t+1\}$, and $n = (t + 1)q + d$, where $q \geq 0$, $0 \leq d \leq t$. Since $n\geq t+1$, we have $q\geq 1$. We have the following claim.
    \begin{claim}\label{clm:cycles}
        We have $(t-1)q\geq d$.
    \end{claim}
    \begin{proof}[Proof of Claim~\ref{clm:cycles}] 
        If $q\geq 2$, then
        \[
        (t-1)q \geq 2(t-1) \geq t \geq d,
        \]
        as desired. Now we can assume that $q=1$. Since $n\neq 2t+1$, we have $d\neq t$. Note that we have $d\leq t$ from the hypothesis. Thus  $d\leq t-1=(t-1)q$, as desired.
    \end{proof}
    If $d\neq 0$, then by  Remark~\ref{rem:reg-for-cycle}, we get $$(t-1)\pd(S/P_t(C_n)) -t = (t-1)(2q+1) - t =(t-1)q + (t-1)q-1 \geq  (t-1)q + d-1= \reg(S/P_t(C_n)),$$
    where the inequality is due to Claim~\ref{clm:cycles}.
    
  If $d=0$, then $n=(t+1)q$, and since $n\neq t+1$, we have $q\geq 2$. Hence by Remark \ref{rem:reg-for-cycle}, we get
  \begin{align*}
      (t-1)\pd(S/P_t(C_n)) -t = (t-1)(2q) - t =(t-1)q +(t-1)q - t &\geq  (t-1)q +2(t-1)-t\\
      &\geq (t-1)q =\reg(S/P_t(C_n)).
  \end{align*}
  Thus, by Lemma~\ref{lem:not-licci-criteria}, we get that $P_t(C_n)_{\mathfrak{m}}$ is not licci in $S_{\mathfrak{m}}$.

To complete the proof, we need to show that if $n=t+1$ or $2t+1$, then $P_t(C_n)_{\mathfrak{m}}$ is licci in $S_{\mathfrak{m}}$. Suppose $n=t+1$. Then given any $i, j \in \{1, \ldots, n\}$ with $i \neq j$, we have $P_t(C_n) \subseteq (x_i, x_j)$ and $P_t(C_n) \not\subseteq (x_i)$. Hence, $\height(P_t(C_n))\geq 2$. Also, from Remark \ref{rem:reg-for-cycle}, we have $\pd(S/P_t(C_n))=2$. Thus, $S/P_t(C_n))$ is Cohen--Macaulay with $\height(P_t(C_n))=2$. By Remark~\ref{rem:hgt2-and3}, it follows that $P_t(C_n)_{\mathfrak{m}}$ is licci in $S_{\mathfrak{m}}$.

  Suppose $n=2t+1$. Then given any given any $i, j \in \{1, \ldots, n\}$ with $i \neq j$, we have $P_t(C_n) \not\subseteq (x_i, x_j)$. Thus, $\height(P_t(C_n))\geq 3$. By Remark~\ref{rem:reg-for-cycle}, we have $\pd(S/P_t(C_n))=3$. Hence, $P_t(G)$ is Cohen--Macaulay of height $3$. By \cite[Theorem 4.13]{AlilooeeFaridiBettiNumbers}, we have $\beta_{3}(S/P_t(C_n)) =1$. Thus, $P_t(C_n)$ is Gorenstein of height $3$. By Remark~\ref{rem:hgt2-and3}, it follows that $P_t(C_n)_{\mathfrak{m}}$ is licci in $S_{\mathfrak{m}}$.
\end{proof}

\bibliographystyle{amsplain}
\bibliography{sample}

@article {path-ideals-CM,
    AUTHOR = {Campos, Daniel and Gunderson, Ryan and Morey, Susan and
              Paulsen, Chelsey and Polstra, Thomas},
     TITLE = {Depths and {C}ohen-{M}acaulay properties of path ideals},
   JOURNAL = {J. Pure Appl. Algebra},
  FJOURNAL = {Journal of Pure and Applied Algebra},
    VOLUME = {218},
      YEAR = {2014},
    NUMBER = {8},
     PAGES = {1537--1543},
      ISSN = {0022-4049,1873-1376},
   MRCLASS = {05E40 (05C05 13A15 13C14 13F55)},
  MRNUMBER = {3175038},
MRREVIEWER = {Siamak\ Yassemi},
       DOI = {10.1016/j.jpaa.2013.12.005},
       URL = {https://doi.org/10.1016/j.jpaa.2013.12.005},
}

@book {HerzogHibibook,
    AUTHOR = {Herzog, J\"urgen and Hibi, Takayuki},
     TITLE = {Monomial ideals},
    SERIES = {Graduate Texts in Mathematics},
    VOLUME = {260},
 PUBLISHER = {Springer-Verlag London, Ltd., London},
      YEAR = {2011},
     PAGES = {xvi+305},
      ISBN = {978-0-85729-105-9},
   MRCLASS = {13D02 (05E40 13D40 13F55 13P10)},
  MRNUMBER = {2724673},
MRREVIEWER = {Rahim\ Zaare-Nahandi},
       DOI = {10.1007/978-0-85729-106-6},
       URL = {https://doi.org/10.1007/978-0-85729-106-6},
}

@article{HunekeUlrich,
    AUTHOR = {Huneke, Craig and Ulrich, Bernd},
     TITLE = {The structure of linkage},
   JOURNAL = {Ann. of Math. (2)},
  FJOURNAL = {Annals of Mathematics. Second Series},
    VOLUME = {126},
      YEAR = {1987},
    NUMBER = {2},
     PAGES = {277--334},
      ISSN = {0003-486X,1939-8980},
   MRCLASS = {13H10 (13D10 13H15 14B07)},
  MRNUMBER = {908149},
MRREVIEWER = {Matthew\ Miller},
       DOI = {10.2307/1971402},
       URL = {https://doi.org/10.2307/1971402}
}

@article{hibi2025complementaryedgeideals,
      title={Complementary edge ideals}, 
      author={Takayuki Hibi and Ayesha Asloob Qureshi and Sara Saeedi Madani},
      year={2025},
      journal={arXiv:2508.09837 [math.AC]}
}

@book {Bruns-Herzog,
    AUTHOR = {Bruns, Winfried and Herzog, J\"urgen},
     TITLE = {Cohen-{M}acaulay rings},
    SERIES = {Cambridge Studies in Advanced Mathematics},
    VOLUME = {39},
 PUBLISHER = {Cambridge University Press, Cambridge},
      YEAR = {1993},
     PAGES = {xii+403},
      ISBN = {0-521-41068-1},
   MRCLASS = {13H10 (13-02)},
  MRNUMBER = {1251956},
MRREVIEWER = {Matthew\ Miller},
}

@article{ficarra2026stablesetassociatedprimes,
      title={The stable set of associated primes of a complementary edge ideal}, 
      author={Antonino Ficarra},
      year={2026},
      journal={arXiv:2603.02358 [math.AC]}
}

@article {Villarreal,
    AUTHOR = {Villarreal, Rafael H.},
     TITLE = {Cohen-{M}acaulay graphs},
   JOURNAL = {Manuscripta Math.},
  FJOURNAL = {Manuscripta Mathematica},
    VOLUME = {66},
      YEAR = {1990},
    NUMBER = {3},
     PAGES = {277--293},
      ISSN = {0025-2611,1432-1785},
   MRCLASS = {13H10 (05C05 05E25 52B20)},
  MRNUMBER = {1031197},
MRREVIEWER = {Aron\ Simis},
       DOI = {10.1007/BF02568497},
       URL = {https://doi.org/10.1007/BF02568497},
}

@article {Peskine-Szpiro,
    AUTHOR = {Peskine, Christian  and Szpiro, Lucien},
     TITLE = {Liaison des vari\'et\'es alg\'ebriques. {I}},
   JOURNAL = {Invent. Math.},
  FJOURNAL = {Inventiones Mathematicae},
    VOLUME = {26},
      YEAR = {1974},
     PAGES = {271--302},
      ISSN = {0020-9910,1432-1297},
   MRCLASS = {14M10},
  MRNUMBER = {364271},
MRREVIEWER = {G.\ Horrocks},
       DOI = {10.1007/BF01425554},
       URL = {https://doi.org/10.1007/BF01425554},
}

@article {HunekeUlrichMonomial,
    AUTHOR = {Huneke, Craig and Ulrich, Bernd},
     TITLE = {Liaison of monomial ideals},
   JOURNAL = {Bull. Lond. Math. Soc.},
  FJOURNAL = {Bulletin of the London Mathematical Society},
    VOLUME = {39},
      YEAR = {2007},
    NUMBER = {3},
     PAGES = {384--392},
      ISSN = {0024-6093,1469-2120},
   MRCLASS = {13C40 (13F20)},
  MRNUMBER = {2331565},
MRREVIEWER = {N.\ Mohan Kumar},
       DOI = {10.1112/blms/bdl008},
       URL = {https://doi.org/10.1112/blms/bdl008},
}

@incollection {Terai,
    AUTHOR = {Terai, Naoki},
     TITLE = {Alexander duality theorem and {S}tanley-{R}eisner rings},
      NOTE = {Free resolutions of coordinate rings of projective varieties
              and related topics (Japanese) (Kyoto, 1998)},
   JOURNAL = {S\=urikaisekikenky\=usho K\B oky\=uroku},
  FJOURNAL = {S\=urikaisekikenky\=usho K\B oky\=uroku},
    NUMBER = {1078},
      YEAR = {1999},
     PAGES = {174--184},
   MRCLASS = {13F55 (13D02)},
  MRNUMBER = {1715588},
MRREVIEWER = {Rafael\ H.\ Villarreal},
}

@book {Eis-book,
    AUTHOR = {Eisenbud, David},
     TITLE = {Commutative algebra},
    SERIES = {Graduate Texts in Mathematics},
    VOLUME = {150},
      NOTE = {With a view toward algebraic geometry},
 PUBLISHER = {Springer-Verlag, New York},
      YEAR = {1995},
     PAGES = {xvi+785},
      ISBN = {0-387-94268-8; 0-387-94269-6},
   MRCLASS = {13-01 (14A05)},
  MRNUMBER = {1322960},
MRREVIEWER = {Matthew\ Miller},
       DOI = {10.1007/978-1-4612-5350-1},
       URL = {https://doi.org/10.1007/978-1-4612-5350-1},
}

@book {MillerSturmfels,
    AUTHOR = {Miller, Ezra and Sturmfels, Bernd},
     TITLE = {Combinatorial commutative algebra},
    SERIES = {Graduate Texts in Mathematics},
    VOLUME = {227},
 PUBLISHER = {Springer-Verlag, New York},
      YEAR = {2005},
     PAGES = {xiv+417},
      ISBN = {0-387-22356-8},
   MRCLASS = {13-01 (05-01 05E99 13D02 14M15 14M25)},
  MRNUMBER = {2110098},
MRREVIEWER = {Joseph\ Gubeladze},
}

@article {AlilooeeFaridi,
    AUTHOR = {Alilooee, Ali and Faridi, Sara},
     TITLE = {On the resolution of path ideals of cycles},
   JOURNAL = {Comm. Algebra},
  FJOURNAL = {Communications in Algebra},
    VOLUME = {43},
      YEAR = {2015},
    NUMBER = {12},
     PAGES = {5413--5433},
      ISSN = {0092-7872,1532-4125},
   MRCLASS = {13D02 (13F20 13F55)},
  MRNUMBER = {3395713},
MRREVIEWER = {Manoj\ Kummini},
       DOI = {10.1080/00927872.2013.837170},
       URL = {https://doi.org/10.1080/00927872.2013.837170},
}

@article{Ficarra-Moradi,
      title={Complementary edge ideals}, 
      author={Antonino Ficarra and Somayeh Moradi},
      year={2025},
      journal={arXiv:2508.10870 [math.AC]}
}

@article {Kimura-Terai-Yoshida,
    AUTHOR = {Kimura, Kyouko and Terai, Naoki and Yoshida, Ken-ichi},
     TITLE = {Licci squarefree monomial ideals generated in degree two or
              with deviation two},
   JOURNAL = {J. Algebra},
  FJOURNAL = {Journal of Algebra},
    VOLUME = {390},
      YEAR = {2013},
     PAGES = {264--289},
      ISSN = {0021-8693,1090-266X},
   MRCLASS = {13F20 (05C65 13C40)},
  MRNUMBER = {3072123},
MRREVIEWER = {N.\ Mohan Kumar},
       DOI = {10.1016/j.jalgebra.2013.06.001},
       URL = {https://doi.org/10.1016/j.jalgebra.2013.06.001},
}

@article {AlilooeeFaridiBettiNumbers,
    AUTHOR = {Alilooee, Ali and Faridi, Sara},
     TITLE = {Graded {B}etti numbers of path ideals of cycles and lines},
   JOURNAL = {J. Algebra Appl.},
  FJOURNAL = {Journal of Algebra and its Applications},
    VOLUME = {17},
      YEAR = {2018},
    NUMBER = {1},
     PAGES = {1850011, 17},
      ISSN = {0219-4988,1793-6829},
   MRCLASS = {13A15 (05A10 05C25 05C38 13D02 13F55)},
  MRNUMBER = {3741068},
MRREVIEWER = {Margherita\ Barile},
       DOI = {10.1142/S0219498818500111},
       URL = {https://doi.org/10.1142/S0219498818500111},
}

@article {polarization-glicci,
    AUTHOR = {Faridi, Sara and Klein, Patricia and Rajchgot, Jenna and
              Seceleanu, Alexandra},
     TITLE = {Polarization and {G}orenstein liaison},
   JOURNAL = {J. Lond. Math. Soc. (2)},
  FJOURNAL = {Journal of the London Mathematical Society. Second Series},
    VOLUME = {112},
      YEAR = {2025},
    NUMBER = {6},
     PAGES = {Paper No. e70319},
      ISSN = {0024-6107,1469-7750},
   MRCLASS = {13C40 (13F55 13P10)},
  MRNUMBER = {5001583},
       DOI = {10.1112/jlms.70319},
       URL = {https://doi.org/10.1112/jlms.70319},
}

@article {Huneke_AJM,
    AUTHOR = {Huneke, Craig},
     TITLE = {Linkage and the {K}oszul homology of ideals},
   JOURNAL = {Amer. J. Math.},
  FJOURNAL = {American Journal of Mathematics},
    VOLUME = {104},
      YEAR = {1982},
    NUMBER = {5},
     PAGES = {1043--1062},
      ISSN = {0002-9327,1080-6377},
   MRCLASS = {13H10 (13D25)},
  MRNUMBER = {675309},
MRREVIEWER = {Hamid\ Rahbar-Rochandel},
       DOI = {10.2307/2374083},
       URL = {https://doi.org/10.2307/2374083},
}

@article {Huneke_Inv,
    AUTHOR = {Huneke, Craig},
     TITLE = {Numerical invariants of liaison classes},
   JOURNAL = {Invent. Math.},
  FJOURNAL = {Inventiones Mathematicae},
    VOLUME = {75},
      YEAR = {1984},
    NUMBER = {2},
     PAGES = {301--325},
      ISSN = {0020-9910,1432-1297},
   MRCLASS = {13D03 (13D25 13H10 14M05 14M10)},
  MRNUMBER = {732549},
MRREVIEWER = {J\"urgen\ Herzog},
       DOI = {10.1007/BF01388567},
       URL = {https://doi.org/10.1007/BF01388567},
}

@article {HunekeUlrich_Duke,
    AUTHOR = {Huneke, Craig and Ulrich, Bernd},
     TITLE = {Algebraic linkage},
   JOURNAL = {Duke Math. J.},
  FJOURNAL = {Duke Mathematical Journal},
    VOLUME = {56},
      YEAR = {1988},
    NUMBER = {3},
     PAGES = {415--429},
      ISSN = {0012-7094,1547-7398},
   MRCLASS = {13H10 (13D10 14M05)},
  MRNUMBER = {948528},
MRREVIEWER = {Matthew\ Miller},
       DOI = {10.1215/S0012-7094-88-05618-9},
       URL = {https://doi.org/10.1215/S0012-7094-88-05618-9},
}

@article {Herzog_Crelle,
    AUTHOR = {Herzog, J\"urgen},
     TITLE = {Deformationen von {C}ohen-{M}acaulay {A}lgebren},
   JOURNAL = {J. Reine Angew. Math.},
  FJOURNAL = {Journal f\"ur die Reine und Angewandte Mathematik. [Crelle's
              Journal]},
    VOLUME = {318},
      YEAR = {1980},
     PAGES = {83--105},
      ISSN = {0075-4102,1435-5345},
   MRCLASS = {13D10 (13H10 14B07)},
  MRNUMBER = {579384},
MRREVIEWER = {Martin\ Lindner},
       DOI = {10.1515/crll.1980.318.83},
       URL = {https://doi.org/10.1515/crll.1980.318.83},
}

@article {KustinMiller_Gorenstein,
    AUTHOR = {Kustin, Andrew R. and Miller, Matthew},
     TITLE = {Deformation and linkage of {G}orenstein algebras},
   JOURNAL = {Trans. Amer. Math. Soc.},
  FJOURNAL = {Transactions of the American Mathematical Society},
    VOLUME = {284},
      YEAR = {1984},
    NUMBER = {2},
     PAGES = {501--534},
      ISSN = {0002-9947,1088-6850},
   MRCLASS = {13D10 (13H10 14B07 14M05)},
  MRNUMBER = {743730},
MRREVIEWER = {J\"urgen\ Herzog},
       DOI = {10.2307/1999093},
       URL = {https://doi.org/10.2307/1999093},
}

@article {KustinMiller_tightGorenstein,
    AUTHOR = {Kustin, Andrew R. and Miller, Matthew},
     TITLE = {Tight double linkage of {G}orenstein algebras},
   JOURNAL = {J. Algebra},
  FJOURNAL = {Journal of Algebra},
    VOLUME = {95},
      YEAR = {1985},
    NUMBER = {2},
     PAGES = {384--397},
      ISSN = {0021-8693},
   MRCLASS = {13H10 (13D10)},
  MRNUMBER = {801274},
MRREVIEWER = {Bernd\ Ulrich},
       DOI = {10.1016/0021-8693(85)90110-3},
       URL = {https://doi.org/10.1016/0021-8693(85)90110-3},
}

@article {KustinMillerUlrich_pureRes,
    AUTHOR = {Kustin, Andrew R. and Miller, Matthew and Ulrich, Bernd},
     TITLE = {Linkage theory for algebras with pure resolutions},
   JOURNAL = {J. Algebra},
  FJOURNAL = {Journal of Algebra},
    VOLUME = {102},
      YEAR = {1986},
    NUMBER = {1},
     PAGES = {199--228},
      ISSN = {0021-8693},
   MRCLASS = {13H10 (13D99)},
  MRNUMBER = {853240},
MRREVIEWER = {J\"urgen\ Herzog},
       DOI = {10.1016/0021-8693(86)90137-7},
       URL = {https://doi.org/10.1016/0021-8693(86)90137-7},
}

@article {Rao_liason,
    AUTHOR = {Prabhakar Rao, Aroor},
     TITLE = {Liaison among curves in {${\bf P}\sp{3}$}},
   JOURNAL = {Invent. Math.},
  FJOURNAL = {Inventiones Mathematicae},
    VOLUME = {50},
      YEAR = {1978/79},
    NUMBER = {3},
     PAGES = {205--217},
      ISSN = {0020-9910,1432-1297},
   MRCLASS = {14M07},
  MRNUMBER = {520926},
MRREVIEWER = {G.\ Horrocks},
       DOI = {10.1007/BF01410078},
       URL = {https://doi.org/10.1007/BF01410078},
}

@article {EneRinaldoTerai_liccibinomial,
    AUTHOR = {Ene, Viviana and Rinaldo, Giancarlo and Terai, Naoki},
     TITLE = {Licci binomial edge ideals},
   JOURNAL = {J. Combin. Theory Ser. A},
  FJOURNAL = {Journal of Combinatorial Theory. Series A},
    VOLUME = {175},
      YEAR = {2020},
     PAGES = {105278, 23},
      ISSN = {0097-3165,1096-0899},
   MRCLASS = {05E40 (13A70)},
  MRNUMBER = {4102156},
MRREVIEWER = {Mehrdad\ Nasernejad},
       DOI = {10.1016/j.jcta.2020.105278},
       URL = {https://doi.org/10.1016/j.jcta.2020.105278},
}

@article {NagelRomer_glicci,
    AUTHOR = {Nagel, Uwe and R\"omer, Tim},
     TITLE = {Glicci simplicial complexes},
   JOURNAL = {J. Pure Appl. Algebra},
  FJOURNAL = {Journal of Pure and Applied Algebra},
    VOLUME = {212},
      YEAR = {2008},
    NUMBER = {10},
     PAGES = {2250--2258},
      ISSN = {0022-4049,1873-1376},
   MRCLASS = {13C40 (05E25 13F55)},
  MRNUMBER = {2426505},
MRREVIEWER = {Rahim\ Zaare-Nahandi},
       DOI = {10.1016/j.jpaa.2008.03.005},
       URL = {https://doi.org/10.1016/j.jpaa.2008.03.005},
}

@article {ConcaDeNegri,
    AUTHOR = {Conca, Aldo and De Negri, Emanuela},
     TITLE = {{$M$}-sequences, graph ideals, and ladder ideals of linear
              type},
   JOURNAL = {J. Algebra},
  FJOURNAL = {Journal of Algebra},
    VOLUME = {211},
      YEAR = {1999},
    NUMBER = {2},
     PAGES = {599--624},
      ISSN = {0021-8693,1090-266X},
   MRCLASS = {13C40 (13H10 13P10)},
  MRNUMBER = {1666661},
MRREVIEWER = {Rafael\ H.\ Villarreal},
       DOI = {10.1006/jabr.1998.7740},
       URL = {https://doi.org/10.1006/jabr.1998.7740},
}

@article {Banerjee2015,
    AUTHOR = {Banerjee, Arindam},
     TITLE = {The regularity of powers of edge ideals},
   JOURNAL = {J. Algebraic Combin.},
  FJOURNAL = {Journal of Algebraic Combinatorics. An International Journal},
    VOLUME = {41},
      YEAR = {2015},
    NUMBER = {2},
     PAGES = {303--321},
      ISSN = {0925-9899,1572-9192},
   MRCLASS = {13F20 (05C10 13D02)},
  MRNUMBER = {3306074},
MRREVIEWER = {Adam\ L.\ Van Tuyl},
       DOI = {10.1007/s10801-014-0537-2},
       URL = {https://doi.org/10.1007/s10801-014-0537-2},
}

@article{path-linear-res,
      title={Trees whose path ideals have linear quotients}, 
      author={Trung Chau and Kanoy Kumar Das and Animikha Dutta Dhar and Pranath S Karanth and Aniruda Suswaram},
      year={2025},
      journal={arXiv:2506.06209 [math.AC]}
}

@article {biCM-Crupi,
    AUTHOR = {Crupi, Marilena and Ficarra, Antonino},
     TITLE = {Edge ideals whose all matching powers are
              bi-{C}ohen-{M}acaulay},
   JOURNAL = {Comm. Algebra},
  FJOURNAL = {Communications in Algebra},
    VOLUME = {54},
      YEAR = {2026},
    NUMBER = {2},
     PAGES = {661--668},
      ISSN = {0092-7872,1532-4125},
   MRCLASS = {13C14 (05E40 13C05 13C15)},
  MRNUMBER = {4996679},
       DOI = {10.1080/00927872.2025.2537272},
       URL = {https://doi.org/10.1080/00927872.2025.2537272},
}

@article {biCM-Gunnar,
    AUTHOR = {Fl{\o}ystad, Gunnar and Vatne, Jon Eivind},
     TITLE = {({B}i-){C}ohen-{M}acaulay simplicial complexes and their
              associated coherent sheaves},
   JOURNAL = {Comm. Algebra},
  FJOURNAL = {Communications in Algebra},
    VOLUME = {33},
      YEAR = {2005},
    NUMBER = {9},
     PAGES = {3121--3136},
      ISSN = {0092-7872,1532-4125},
   MRCLASS = {13F55 (14F05)},
  MRNUMBER = {2175383},
MRREVIEWER = {Anargyros\ Katsabekis},
       DOI = {10.1081/AGB-200066131},
       URL = {https://doi.org/10.1081/AGB-200066131},
}

@article {biCM-Herzog,
    AUTHOR = {Herzog, J\"urgen and Rahimi, Ahad},
     TITLE = {Bi-{C}ohen-{M}acaulay graphs},
   JOURNAL = {Electron. J. Combin.},
  FJOURNAL = {Electronic Journal of Combinatorics},
    VOLUME = {23},
      YEAR = {2016},
    NUMBER = {1},
     PAGES = {Paper 1.1, 14},
      ISSN = {1077-8926},
   MRCLASS = {05E40 (13C14 13F20 13F55)},
  MRNUMBER = {3484706},
MRREVIEWER = {Jin\ Guo},
       DOI = {10.37236/5557},
       URL = {https://doi.org/10.37236/5557},
}

@article {Gitler-Valencia,
    AUTHOR = {Gitler, Isidoro and Valencia, Carlos E.},
     TITLE = {Bounds for invariants of edge-rings},
   JOURNAL = {Comm. Algebra},
  FJOURNAL = {Communications in Algebra},
    VOLUME = {33},
      YEAR = {2005},
    NUMBER = {5},
     PAGES = {1603--1616},
      ISSN = {0092-7872,1532-4125},
   MRCLASS = {13A99 (05C69)},
  MRNUMBER = {2149079},
MRREVIEWER = {Ioannis\ Emmanouil},
       DOI = {10.1081/AGB-200061033},
       URL = {https://doi.org/10.1081/AGB-200061033},
}

@article {localization-CM,
    AUTHOR = {Matijevic, Jacob and Roberts, Paul},
     TITLE = {A conjecture of {N}agata on graded {C}ohen-{M}acaulay rings},
   JOURNAL = {J. Math. Kyoto Univ.},
  FJOURNAL = {Journal of Mathematics of Kyoto University},
    VOLUME = {14},
      YEAR = {1974},
     PAGES = {125--128},
      ISSN = {0023-608X},
   MRCLASS = {13E05},
  MRNUMBER = {340243},
MRREVIEWER = {Melvin\ Hochster},
       DOI = {10.1215/kjm/1250523283},
       URL = {https://doi.org/10.1215/kjm/1250523283},
}

\end{document}